\documentclass[leqno]{macrorend}

\volumeyear{2012}\yearnumber{?}\volumenumber{70}

\usepackage{graphicx}
\usepackage{booktabs}
\usepackage{mathrsfs}

\newtheorem*{theorem*}{Theorem}
\newcommand{\jj}{j} 
\newcommand{\ekk}{k} 
\newcommand{\epp}{P}

\newcommand{\infinity}{\lower1pt\hbox{\large$\infty$}}
\newcommand{\sC}{\mathscr{C}}
\newcommand{\sQ}{\mathscr{Q}}
\newcommand{\CP}{\mathbb{CP}}
\newcommand{\smalpha}{\hbox{\small$\alpha$}}
\newcommand{\smbeta}{\hbox{\small$\beta$}}
\newcommand{\smeta}{\hbox{\small$\eta$}}
\newcommand{\smpi}{\hbox{\small$\pi$}}
\newcommand{\smomega}{\hbox{\small$\omega$}}

\begin{document}

\selectlanguage{english}

\articolo[Cubic surfaces]{Twistor lines on cubic
  surfaces\footnotemark[1]}{J.~Armstrong, M.~Povero and S.~Salamon}

\footnotetext[1]{Based on the third author's talk on 23 March 2012 at
  the conference {\it Geometria delle Variet\`a Algebriche} in Turin,
  dedicated to Alberto Conte}

\begin{abstract}
It is shown that there exist non-singular cubic surfaces in $\CP^3$
containing $5$ twistor lines. This is the maximum number of twistor
fibres that a non-singular cubic can contain. Cubic surfaces in
$\CP^3$ with $5$ twistor lines are classified up to transformations
preserving the conformal structure of $S^4$.
\end{abstract}

\section*{Introduction}

The twistor space, $Z$, of an oriented Riemannian $4$-manifold $M$ is
the bundle of almost-complex structures on $M$ compatible with the
metric and orientation. The $6$-dimensional total space of the twistor
space admits a canonical almost-complex structure which is integrable
whenever the $4$-manifold is half conformally flat.

The definition of the twistor space does not require the full
Riemannian metric; it only depends upon the conformal structure of the
manifold. The idea of studying the twistor space is that, on a half
conformally flat manifold, the conformal geometry of $M$ is encoded
into the complex geometry of $Z$.

As an example, the condition that an almost-complex structure on $M$
compatible with the conformal structure is \emph{integrable} can be
interpreted as saying that the corresponding section of $Z$ defines a
\emph{holomorphic} submanifold of $Z$.

The basic example of twistor space is that of the $4$-sphere $S^4$,
which itself may be identified with the quaternionic projective line
$\mathbb{HP}^1$, and is topologically $\RR^4\cup\infinity$. The twistor
space in this case is biholomorphic to $\CP^3$, and the associated
bundle structure $\CP^3\to\mathbb{HP}^1$ is the Hopf
fibration. Following the work of Penrose, Ward and Atiyah, it was used
to great effect in classifying instanton bundles on $S^4$ \cite{AW,ADHM}.

Combining these two facts, we see that complex hypersurfaces in
$\CP^3$ locally give rise to integrable complex structures on $S^4$
compatible with the metric. For topological reasons there are no
global almost-complex structures on $S^4$, so no hypersurface in
$\CP^3$ can intersect every fibre of the Hopf fibration in exactly one
point.

One can try to investigate the algebraic geometry of surfaces in
$\CP^3$ from this twistor perspective. In this paper, we take the opportunity
to revisit some of the beautiful results on cubic surfaces discovered
by geometers in the nineteenth century. A brief history of their
discoveries can be found in \cite{dolgachev}.

A natural question when studying complex surfaces from this point of
view is to classify surfaces in $\CP^3$ of degree $d$ up to a
conformal transformation of the base space $S^4$. Various conformal
invariants of a surface can be defined immediately. The fibres of the
Hopf fibration are complex projective lines in $\CP^3$, and the number
of fibres that lie entirely in the suface is an invariant of the
surface up to conformal transformation.

Closely related invariants arise from the topology of the discriminant
locus. A generic fibre, intersecting the surface transversely, will
contain $d$ points. This is simply because the defining polynomial of
the surface, when restricted to the fibre, gives a polynomial of
degree $d$. The set of points where the discriminant of this
polynomial vanishes is called the discriminant locus. It can be
thought of as the set of fibres that are not transverse to the surface
at each point of intersection, or as the set of points of $S^4$ where
we cannot locally define a complex structure corresponding to the
surface.

In \cite{salamonViaclovsky}, quadric surfaces are classified up to
conformal transformation in considerable detail. The table below shows
all the possible topologies of the discriminant locus in this case and
how they correspond to the number of twistor lines. A preliminary
question to ask when trying to study the conformal geometry of
surfaces of higher degree is: what is the maximum number of twistor
lines on surfaces of that degree?

\medskip
\begin{table}[h]
\centering{
\begin{tabular}{cl} \toprule
No of twistor lines & Topology of discriminant locus \\ \midrule $0$ &
Torus \\ $1$ & Torus with two points pinched together \\ $2$ & Torus
with two pairs of points pinched together \\ $\infinity$ & Circle
\\ \bottomrule
\end{tabular}
}
\end{table}
\bigskip

Before restricting to twistor lines, it is worth reviewing pure
algebro-geometric results on the number of projective lines on a
surface of given degree. Since twistor lines are fibres of a fibration
they must be skew (i.e., mutually disjoint), so we will also review
the maximum number of skew lines on a surface of degree $d$.

Dimension counting alone leads one to expect that a quadric surface
will contain an infinite number of lines, a cubic surface a finite
number of lines and a higher degree surface will generically contain
no lines at all.

A startling result is the celebrated Cayley--Salmon theorem: all
non-singular cubic surfaces contain precisely $27$ lines. Moreover a
non-singular cubic surface contains precisely $72$ sets of $6$ skew
lines.

The situation for higher degree curves is less well understood. The
state of knowledge about the number of lines on surfaces of degree $d$
was both reviewed and advanced in \cite{sarti}. We summarize these
findings next.

Define $N_d$ to be the maximum number of lines on a smooth projective
surface of degree $d$. Then:
\begin{itemize}\itemsep0pt\itemindent0pt
\item\ there are always 27 lines on a cubic,
\item\ $N_4 = 64$ (see \cite{segre}),
\item\ $N_d \leqslant (d-2)(11 d - 6 )$ (see \cite{segre}),
\item\ $N_d \geqslant 3 d^2$ (see \cite{caporasoHarrisMazur}),
\item\ $N_6 \geqslant 180$, $N_8 \geqslant 352$, $N_{1}2\geqslant 864$, $N_{20} \geqslant
  1600$ (see \cite{caporasoHarrisMazur,sarti}).
\end{itemize}

\smallbreak

\noindent Here are the bounds on $S_d$, the maximum number of skew
lines:
\begin{itemize}\itemsep1pt\itemindent0pt
\item\ there are always 6 skew lines on a cubic,
\item\ $S_4=16$ (see \cite{nikulin}),
\item\ $S_d \leqslant 2 d (d-2)$ when $d \geqslant 4$ (see \cite{miyaoka}),
\item\ $S_d \geqslant d(d-2) + 2$ (see \cite{rams}),
\item\ $S_d \geqslant d(d-2) + 4$ when $d\geqslant7$ is odd (see
  \cite{sarti}).
\end{itemize}
\medbreak

Specializing to the case of twistor lines, it was noted in the first
version of \cite{salamonViaclovsky} that the number of twistor lines
is at most $d^2$ when $d \geqslant 3$ and that there exists a quartic
surface containing exactly 8 twistor lines.

In this paper, we determine the maximum number of twistor lines on a
smooth cubic surface. We shall show that in fact there are at most $5$
twistor lines, and we shall give a detailed classification of all
cubic surfaces with $5$ twistor lines. In particular we shall prove
the

\begin{theorem*}
  Any set of $5$ points on a $2$-sphere, no $4$ of which lie on a
  circle, determines a one-parameter family of non-singular cubic
  surfaces with $5$ twistor lines. All cubics in the family are
  projectively, but not conformally, equivalent. Two such cubic
  surfaces are projectively equivalent if and only if the sets of $5$
  points on the $2$-sphere are conformally equivalent. All cubic
  surfaces with $5$ twistor lines arise in this way.
\end{theorem*}

One would like explicit examples of such surfaces. We provide the
necessary formulae and find the most symmetrical examples. In
particular, we shall show that the cubic surface with $5$ twistor
lines which has the largest conformal symmetry group is projectively,
but not conformally, equivalent to the Fermat cubic. There are various
choices one can make for a twistor structure on $\CP^3$ that give the
Fermat cubic $5$ twistor lines, and the set of such structures has
$54$ connected components.

The paper begins with a brief review of the twistor fibration of
$\CP^3$ and then moves on to discuss cubic surfaces. We review the
classical results on cubic surfaces and demonstrate how the same ideas
can be used to prove results about the twistor geometry.

\section{The twistor fibration}

To identify $S^4$ with $\mathbb{HP}^1$, we define two equivalence
relations on $\HH\times\HH$:
\[\begin{array}{l} 
[q_1, q_2 ]\sim_{\HH} [\lambda q_1, \lambda q_2], \qquad\lambda\in\HH^*,\\[8pt]
[q_1, q_2 ]\sim_{\CC} [\lambda q_1, \lambda q_2], \qquad\lambda\in\CC^*.
\end{array}\]
By definition, the quotient of $\HH \times \HH$ by the first equivalence
relation is the quaternionic projective line. Since $\HH \times
\HH \cong \CC^4$, the quotient by the second relation is
isomorphic to the complex projective space $\CP^3$.

Thus we can define a map $\smpi:\CP^3\to S^4$ by sending a complex
$1$-dimensional subspace of $\CC^4$ to the quaternionic
$1$-dimensional subspace of $\HH^2$ that it spans. The map $\smpi$ is
equivalent to the more general twistor fibration defined on an
arbitrary oriented Riemannian $4$-manifold as the total space of the
bundle of almost-complex structures compatible with the metric and
orientation.

On any twistor fibration one can define a map $\jj$ which sends an
almost-complex structure $J$ to $-J$. In our case, applying $\jj$ can
be thought of as the action of multiplying a $1$-dimensional complex
subspace of $\CC^4$ by the quaternion $j$ in order to get a new
$1$-dimensional subspace.

The map $\jj$ is an anti-holomorphic involution of the twistor space
to itself with no fixed points. Starting with such a map $\jj$, one
can recover the twistor fibration: given a point $z$ in $\CP^3$ there
is a unique projective line connecting $z$ and $\jj(z)$. These lines
form the fibres. We will call an anti-holomorphic involution on
$\CP^3$ obtained by conjugating $j$ by a projective transformation a
\emph{twistor structure}. The standard twistor structure on $\CP^3$ is
given by
\[[z_1,z_2,z_3,z_4]\longmapsto
[-\overline{z_2},\,\overline{z_1},-\overline{z_4},\,\overline{z_3}].\]
\smallbreak

The conformal symmetries of $S^4$ can be represented by quaternionic
M\"obius transformations 
\[ q\longmapsto (qc+d)^{-1}(qa+b),\qquad q\in\HH\] 
(see, for example, \cite{GSS}). These correspond to the projective
transformations of $\CP^3$ that preserve $\jj$. Thus we will say that
two complex submanifolds of $\CP^3$ are \emph{conformally equivalent}
if they are projectively equivalent by a transformation that preserves
$\jj$.

As an example, consider lines in $\CP^3$. If both lines are fibres of
$\smpi$ then they are conformally equivalent by an isometry of $S^4$
sending the image of one line under $\smpi$ to the image of the other
line. If a line is not a fibre of $\smpi$ then its image will be a
round $2$-sphere in $S^4$ (corresponding to a $2$-sphere or a
$2$-plane in $\RR^4$). Given such a $2$-sphere in $S^4$, there are in
fact two projective lines lying above it in $\CP^3$. Therefore, a line
in $\CP^3$ is given by either an oriented $2$-sphere or a point in
$S^4$. Moreover, any two such $2$-spheres are conformally
equivalent. This geometrical correspondence is described in detail by
Shapiro \cite{shapiro}.

As another example, consider planes in $\CP^3$. A plane in $\CP^3$
cannot be transverse to every fibre of $\smpi$ because it would then
define a complex structure on the whole of $S^4$, which is a
topological impossibility. Thus a plane always contains at least one
twistor fibre. Twistor fibres are always skew, whereas two lines in a
plane always meet. Therefore a plane always contains exactly one
twistor fibre. If one picks another line in the plane transverse to
the fibre, its image under $\smpi$ will be a $2$-sphere. We can find a
conformal transformation of $S^4$ mapping any $2$-sphere with a marked
point to any other $2$-sphere with a marked point. We deduce that any
two planes in $\CP^3$ are conformally equivalent.

The case of quadric surfaces is considered in detail in
\cite{salamonViaclovsky} and is much more complicated. The aim of this
paper is to make a start on the analogous question for cubic surfaces.

\section{The Schl\"afli graph}

Before looking at the twistor geometry of cubic surfaces. Let us
review the classical results about the lines on twistor surfaces.

The Cayley--Salmon theorem states that every non-singular cubic
surface contains exactly 27 lines \cite{cayley}. Schl\"afli discovered
that the intersection properties of these 27 lines are the same for
all cubics \cite{schlafli}. This means that we can define the
Schl\"afli graph of a cubic surface to be a graph with 27 vertices
corresponding to each line on the cubic and and with an edge between
the two vertices whenever the corresponding lines {\em do not}
intersect. This graph will be independent of the choice of
non-singular cubic surface. This definition is the standard one used
by graph theorists, but from our point of view the complement of the
Schl\"afli graph showing which lines \emph{do} intersect is more
natural. It is shown in Figure~\ref{fig:schlafliGraph1}

\begin{figure}[h]
\scalebox{1.3}{\includegraphics{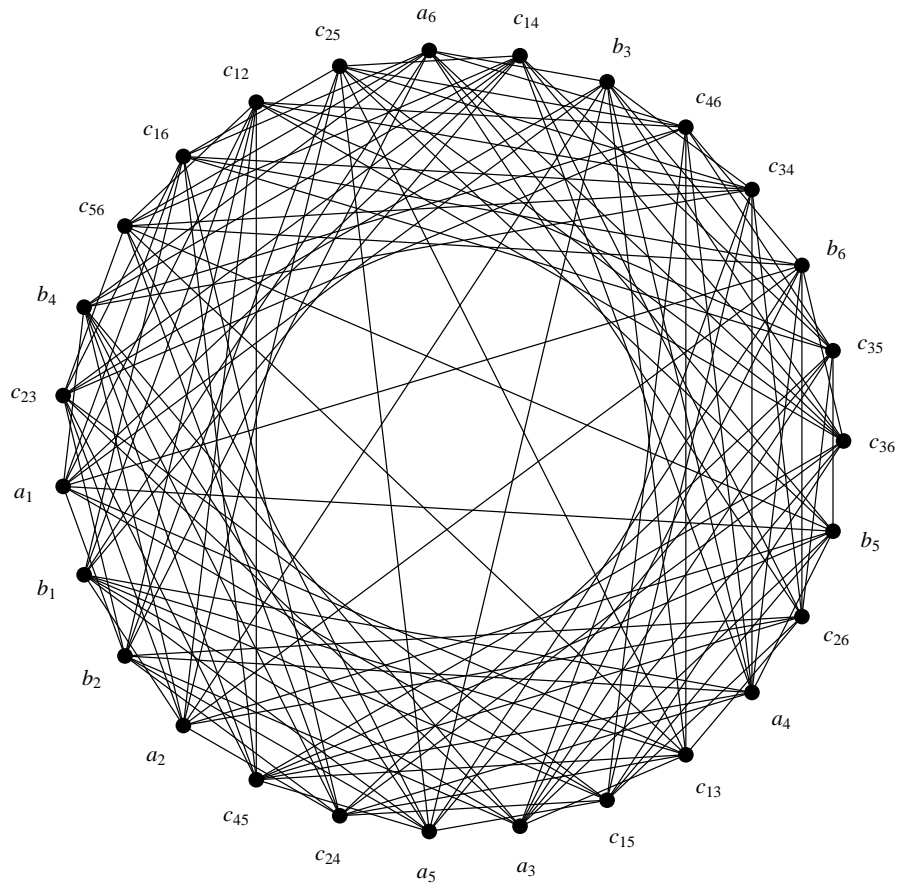}}
 \centering \caption {The complement of
  the Schlafli graph emphasizing a symmetry of order 9}
\label{fig:schlafliGraph1}
\end{figure}

Understanding the Schl\"afli graph provides a good deal of insight
into the geometry of cubic surfaces. It is interesting from a purely
graph theoretic point of view. Among its many properties, one
particularly nice one is that it is $4$-\emph{ultrahomogeneous}. A
graph is said to be $k$-ultrahomogeneous if every isomorphism between
subgraphs with at most $k$ vertices extends to an automorphism of the
entire graph.  If a graph is $5$-ultrahomogenous it is
$k$-ultrahomogeneous for any $k$. It turns out that the Schl\"afli
graph and its complement are the only $4$-ultrahomogeneous connected
graphs that are not $5$-ultrahomogeneous \cite{graphTheoryRef}.

Although our picture of the Schl\"afli graph is pretty, it is not very
practical. Schl\"afli devised a notation that allows one to understand
the graph more directly, and we shall now describe this.

Among the $27$ lines one can always find a set of $6$ skew lines. We
label these $a_1$, $a_2$, $a_3$, $a_4$, $a_5$ and $a_6$. Having chosen
these labels, there will now be $6$ more skew lines $b_i$ ($1
\leqslant i\leqslant 6$) with each $b_i$ intersecting all of the $a$
lines except for $a_i$. We thereby obtain the configuration shown in
Figure~\ref{fig:aAndBLines}.

\smallskip
\begin{figure}[ht]
\scalebox{1.1}{\includegraphics{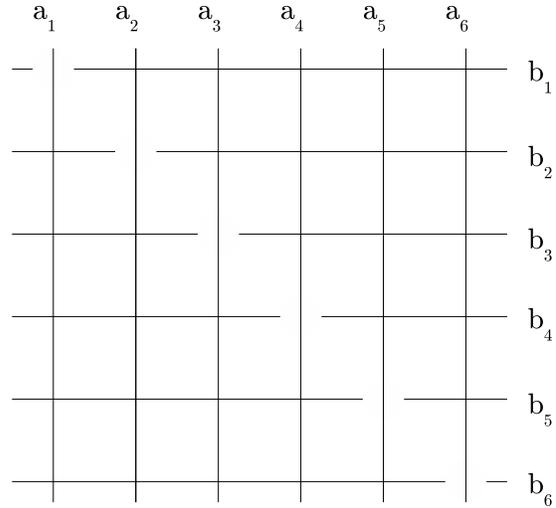}}
\centering \caption {A ``double six''}
\label{fig:aAndBLines}
\end{figure}
\smallskip

The remaining $15$ lines are labelled $c_{ij}$ with ($1 \leqslant i<j
\leqslant 6$). The line $c_{ij}$ is defined by the property that it
intersects $a_i$ and $a_j$, but no other $a$ lines. The full
intersection rules for distinct lines on the cubic surface are:
\begin{itemize}\itemsep2pt\itemindent0pt
\item $a_i$ never intersects $a_j$.
\item $a_i$ intersects $b_j$ iff $i \ne j$.
\item $a_i$ intersects $c_{jk}$ iff $i \in \{j,k\}$.
\item $b_i$ never intersects $b_j$.
\item $b_i$ intersects $c_{jk}$ iff $i \in \{j,k\}$.
\item $c_{ij}$ intersects $c_{kl}$ iff $\{i,j\} \cap \{k,l\} =
  \varnothing$.
\end{itemize}
\smallbreak

Another graphical representation of these properties was given in
\cite{povero} and is reproduced in
Figure~\ref{fig:schlafliRoadmap}. This is intended to be used rather
like the tables of distances between towns that are used to be found
in road atlases. A red/darker square indicates that the two lines
intersect and a cyan/lighter ones indicates that the lines are
skew. In this case we have chosen the ordering of the lines to show
that this figure can be constructed using only a small number of
different types of tile of size $3\times3$. The grouping is indicated
with black lines.

\begin{figure}[ht]
\scalebox{.97}{\includegraphics{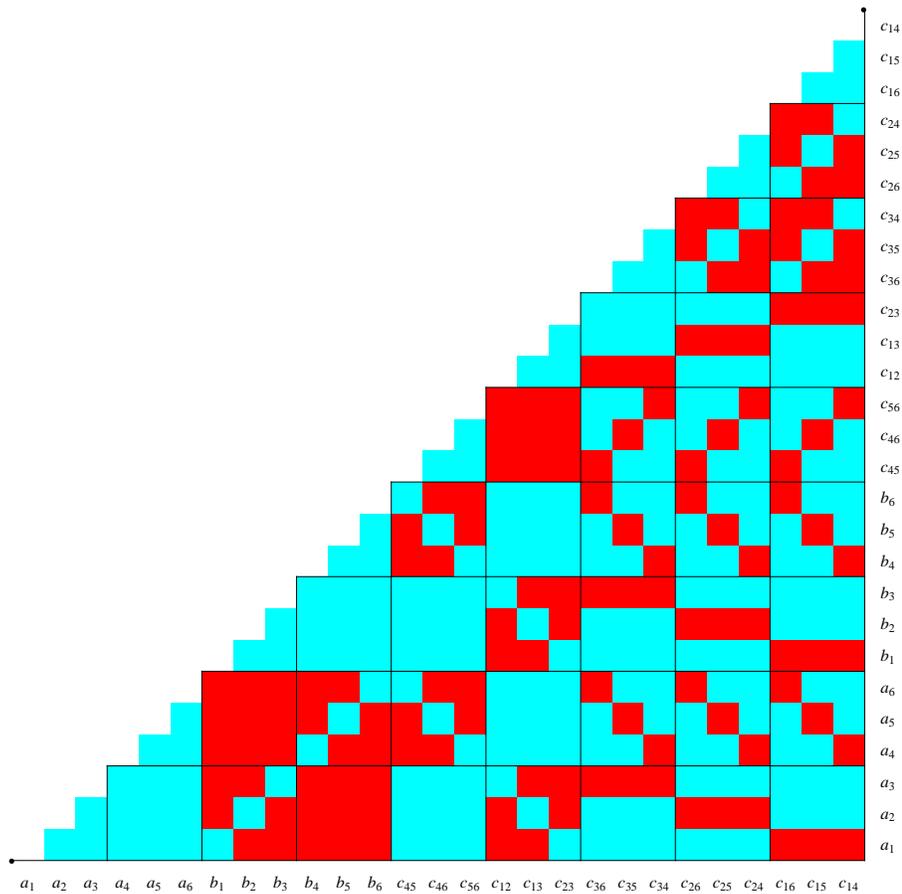}} \centering \caption {A
    tabular representation of the Sch\"afli graph}
\label{fig:schlafliRoadmap}
\end{figure}

\bigbreak

Being 4-ultrahomogeneous, it is no surprise that the Schl\"afli graph
has a large group of automorphisms. Each choice of $6$ skew lines from
the $27$ will give us a different way of labelling the lines as $a_i$,
$b_i$ and $c_{ij}$. Schl\"afli used the following notation for each
choice:
\[
\left( \begin{array}{cccccc} a_1 & a_2 & a_3 & a_4 & a_5 & a_6 \\ b_1
  & b_2 & b_3 & b_4 & b_5 & b_6
\end{array}
\right)
\]
A set of 12 lines with these intersection properties is called a
``double-six''. In this notation here are the forms of the other
double sixes:
\[
\left( \begin{array}{cccccc} a_1 & a_2 & a_3 & c_{56} & c_{46} &
  c_{45} \\ c_{23} & c_{13} & c_{12} & b_4 & b_5 & b_6
\end{array}
\right)
\]
\[
\left( \begin{array}{cccccc} a_1 & b_1 & c_{23} & c_{24} & c_{25} &
  c_{26} \\ a_2 & b_2 & c_{13} & c_{14} & c_{15} & c_{16}
\end{array}
\right)
\]
By varying the indices this gives a total of 36 double sixes on any
cubic surface --- hence 72 choices of a set of six disjoint lines and
$72\cdot 6!=51840$ automorphisms of the Schl\"afli graph.

This large automorphism group is in fact the Weyl group of the
exceptional Lie group $E_6$, and $27$ is the dimension of the smallest
non-trivial irreducible representation of $E_6$. Relative to the
isotropy subgroup $SU(6)\times_{\ZZ_2}SU(2)$ of the corresponding Wolf
space \cite{wolf}, this representation decomposes into irreducible
subspaces of dimension $12$ and $15$, namely $\CC^2\otimes\CC^6$ and
$\raise1.5pt\hbox{$\bigwedge^{\!2}$}\overline{\CC^6}$. This is the
algebraic interpretation of a double six.

As Schl\"afli discovered, consideration of the arrangement of the 27
lines on a non-singular cubic surface rapidly leads to a
classification of cubic surfaces up to projective transformation
\cite{schlafli}. By applying the same ideas, we can find a similar
classification of cubic surfaces with sufficiently many twistor lines.

\section{Classifying cubic surfaces with $5$ twistor lines}

As a first application of the Schl\"afli graph to the study of
twistor lines on cubic surfaces we prove

\begin{lemma}
If a non-singular cubic surface in $\CP^3$ contains four twistor lines
$a_1,a_2,a_3,a_4$ then, in Schl\"afli's notation, $\jj b_5 = b_6$.
\end{lemma}
\begin{proof}
Since the Schl\"afli graph is $4$ ultrahomogenoeus and twistor lines
are always skew, we can assume that the first four lines are indeed
those of a double six.

In Schl\"afli's notation, the line $b_5$ intersects $a_1$, $a_2$,
$a_3$ and $a_4$. Therefore $\jj b_5$ intersects $\jj a_1=a_1$, $\jj
a_2=a_2$, $\jj a_3=a_3$ and $\jj a_4=a_4$. Since it $\jj b_5$ is a
line and since it intersects the cubic surface in $4$ points, it must
lie in the cubic surface. Since $\jj$ has no fixed points, the points
of intersection of $b_5$ and $\jj b_5$ with the line $a_1$ must be
distinct. So $\jj b_5\ne b_5$, Given this and the fact that it
intersects $a_1$, $a_2$, $a_3$ and $a_4$ we deduce that $\jj b_5 =
b_6$.
\end{proof}

\begin{corollary}\label{c56}
  If a non-singular cubic surface contains five twistor lines, and we
  label the first four $a_1,a_2,a_3,a_4$ in Schl\"afli's notation then
  the fifth line is $c_{56}$.
\end{corollary}

\begin{proof}
Since the fifth twistor line must be skew to $a_1$, $a_2$, $a_3$ and
$a_4$, it must be one of $a_5$, $a_6$, $c_{56}$. Suppose that the
fifth line is $a_6$. This means it intersects $b_5$ so $\jj a_6=a_6$
intersects $\jj b_5=b_6$, which is a contradiction. The same argument
shows that the fifth line cannot be $a_5$.\end{proof}

The arrangement of lines described by Corollary~\ref{c56} is
illustrated in Figure~\ref{fig:sevenLines}. The twistor lines are
shown as roughly vertical.

\begin{figure}[ht]
\scalebox{1.3}{\includegraphics{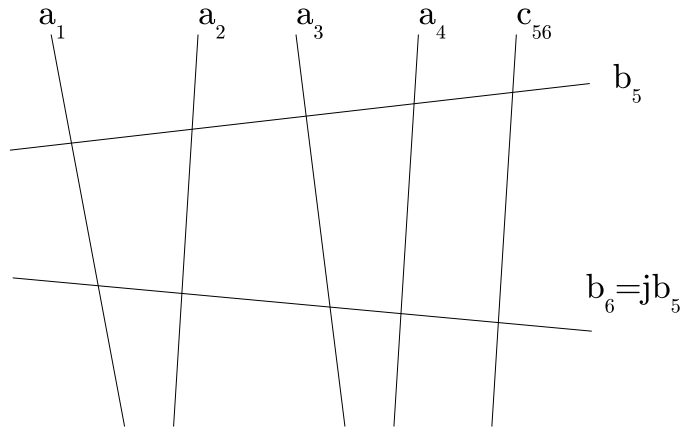}} 
\centering \caption {Seven lines}
\label{fig:sevenLines}
\end{figure}
\medbreak

In particular we have proved:

\begin{theorem}
A non-singular cubic surface contains at most five twistor lines.
\end{theorem}

This raises the question of whether or not we can find cubic surfaces
containing $5$ twistor lines. Simple dimension counting suggests it
should be easy to find cubic surfaces which contain $4$ twistor
lines. Simply select any four twistor lines and apply the well-known

\begin{proposition}
Four lines in $\CP^3$ always lie on a (possibly singular) cubic
surface.
\end{proposition}

\begin{proof}
Pick $4$ points on each line to get a total of 16 points. If a cubic
surface has $4$ points in common with a line, then it contains the
entire line. So if we can find a cubic containing all 16 points, it
will contain all $4$ lines.

The general equation for a cubic surface has $20$ coefficients, since
this is the dimension of $S^3(\CC^4)$. Putting the coordinates of
these 16 points into the equation for the cubic surface gives us 16
linear equations in the 20 unknown coefficients, so non-trivial
solutions exist. \end{proof}

It will become clear later that if we choose everything generically,
the cubic surface will be non-singular. A corollary of this is that
{\,\it four generic lines in $\CP^3$ have two lines intersecting all
  four of them}. This observation can be proved easily enough without
appealing to the theory of cubics --- for example one can use 2-forms
to represent points of the Klein quadric, or Schubert
calculus. Indeed, in \cite{geometryAndTheImagination} this observation
is used as the starting point to establish the existence of the 27
lines on a cubic!

The same dimension counting argument tells us that $5$ lines (let
alone twistor ones) do not generically all lie on a cubic surface. Let
us understand geometrically when $5$ lines \emph{do} all lie on a cubic
surface.

\begin{proposition}\label{5+1}
  Five lines in $\CP^3$ lie on a (possibly singular) cubic surface if
  they are collinear, that is, there exists a fifth line intersecting
  all four.
\end{proposition}

\begin{proof}
  Let $\ell_1,\ldots,\ell_5$ be the lines and let $\ekk$ be another
  line intersecting all five in the points $p_i$. Choose 3 other
  points on each of the lines to get a set $\epp$ of 20 points.

The condition on the coefficients of a cubic surface for it to contain
all the points in $\epp$ except for $p_5$ is represented by $19$
linear equations in $20$ unknowns. So we can find a cubic surface
passing through all the points marked in black in
Figure~\ref{fig:fiveLines}. This cubic surface has $4$ points in common
with $\ekk$ so it contains $\ekk$. In particular it contains $p_5$. So
it actually contains all 5 of the $\ell$ lines.
\end{proof}

\begin{figure}[h]
\scalebox{1.1}{\includegraphics{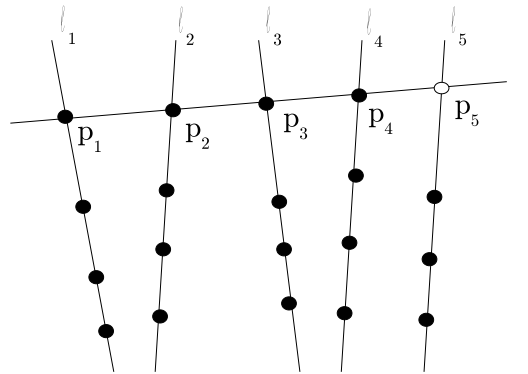}} 
\centering 
\caption {Nineteen points plus one}
\label{fig:fiveLines}
\end{figure}

As a partial converse to Proposition~\ref{5+1}, we remark that if $5$
\emph{skew} lines lie on a cubic surface then they are necessarily
collinear. For example, if the cubic is non-singular then the
Schl\"afli graph guarantees that the $5$ lines are collinear; we can
label them $a_1,\ldots,a_5$ (all intersecting just $a_6$) or
$a_1,a_2,a_3,a_4,c_{56}$ (all intersecting $b_5$ and $b_6$).

The situation that is of interest for us is the second, in which we
have $5$ lines that are collinear in two different ways. The dimension
counting argument above now shows that if one has $5$ lines that are
collinear in two ways, there will be a one-parameter family of (possibly
singular) cubic surfaces containing all the lines.

Another way of seeing why there is a one-parameter family of cubics
through such a configuration of lines is to observe that there is a
one-parameter family of projective transformations that fixes all the
lines. To see this observe that you can choose coordinates such that
$b_5$ is given by the equations $z_1=z_3=0$ and $b_6$ is given by the
equations $z_2=z_4=0$. Since the lines $a_1$, $a_2$, $a_3$, $a_4$,
$c_{56}$ are all skew, their intersections with $b_5$ are distinct. So
we can make a M\"obius transformation of $z_2$ and $z_4$ so that in
the inhomogenous coordinate $z_2/z_4$ the intersection points of
$a_1$, $a_2$, $c_{56}$ with $b_5$ are respectivly $0$, $1$,
$\infinity$. Similarly we can choose our coordinates such that the
intersections of $a_1$, $a_2$, $c_{56}$ with $b_6$ correspond to
$z_1/z_3=0$, $1$, $\infinity$. With these specifications, one can choose
to independently rescale the coordinate pairs $(z_1,z_3)$, $(z_2,z_4)$
and one will still have coordinates with these properties. This choice
of coordinates gives us a one-parameter family of projective
transformations that fix all the lines.

Given two non-singular cubics $(\sC_1,\sC_2)$ that each contain all 7
lines, we can construct the projective transformation mapping $\sC_1$ to
$\sC_2$ directly from the geometry of the cubics.

Indeed, given a point $p \in \CP^3$ away from $b_5$ and $b_6$ there is
a unique line $\ell_p$ passing through $b_5$, $b_6$ and $p$. This line
intersects each $\sC_i$ in $3$ points, so for generic $p$ there is a
unique projective transformation of $\ell_p$ fixing the points where
$\ell_p$ intersects $b_5$ and $b_6$, and mapping the remaining point
of $\ell_p\cap\sC_1$ to that of $\ell_p\cap\sC_2$. If we define $\Phi$
to map $p$ to the image of $p$ under this projective transformation,
then we see that, so long as it is defined, $\Phi$ maps $\sC_1$ to
$\sC_2$. If we can show that $\Phi$ extends to a biholomorphism, then
we will have shown that $\Phi$ is a projective transformation. This is
not too difficult to prove directly, but we will postpone the proof to
the next section when it falls out from general theory.

We have just shown that any two non-singular cubics that contain all 7
lines will be projectively equivalent by a projective transformation
that fixes all $7$ lines.

We have already seen that there is only a one-parameter family of
projective transformations that fix all the lines, so there is at most
a parameter family of non-singular cubics containing all 7
lines. Since non-singularity is an open condition on the space of
cubic surfaces, there is at most a one-parameter family of cubics
containing all $7$ lines if there are any non-singular cubics
containing all $7$ lines.

Putting all of this information together, we end up with a
classification of non-singular cubic surfaces. To make things
explicit, write $(l_1, l_2; l_3, l_4)_k$ for the cross ratio of the
intersection points of four lines $l_i$ meeting on a fifth line
$k$. We can then define four invariants associated with the
configuration of lines as follows:
\begin{equation}
\label{crossRatioInvariants}
\begin{array}{lcl}
 \smalpha & = & (c_{56}, a_1; a_2, a_3)_{b_5} \\[2pt]
 \smalpha' & = & (c_{56}, a_1; a_2, a_3)_{b_6} \\[2pt]
 \smbeta & = & (c_{56}, a_1; a_2, a_4)_{b_5} \\[2pt] 
\smbeta' & = & (c_{56}, a_1; a_2, a_4)_{b_6}
\end{array}
\end{equation}

With this notation, we state

\begin{theorem}
If $b_5$ and $b_6$ are skew lines in $\CP^3$ and $a_1$, $a_2$, $a_3$,
$a_4$, $c_{56}$ are five other skew lines each passing through $b_5$
and $b_6$ then consider the pencil of cubics spanned by the following
two polynomials in $z_1$, $z_2$, $z_3$, $z_4$:
\begin{eqnarray}
\label{singularCubic1}
\begin{array}{cccrc}
\sC_1 &=&& \left[(-\smalpha \smbeta \smbeta '+\smbeta^2 \smbeta '+\smbeta
\smalpha ' \smbeta '-\smbeta^2 \smalpha ' \smbeta
'-\smbeta (\smbeta ')^2+\smalpha \smbeta (\smbeta ')^2) z_3\right]
\kern-10pt & z_2^2 \\[3pt] 
&& +& \left[ \left(\smalpha \smbeta \smalpha '-\smbeta^2
  \smalpha '-\smalpha \smalpha ' \smbeta '+\smbeta^2 \smalpha
  ' \smbeta '+\smalpha (\smbeta ')^2-\smalpha \smbeta (\smbeta
  ')^2 \right) z_1 \right. & \\ 
&&& \ +\left. \left(-\smbeta\smalpha'
  +\smbeta^2 \smalpha '+\smalpha \smbeta '-\smbeta^2 \smbeta
  '-\smalpha (\smbeta ')^2+\smbeta (\smbeta ')^2\right)
  z_3\right] 
& \kern-10pt z_2 z_4 \\[3pt] 
&& + & \left[(\smbeta \smalpha '-\smalpha \smbeta
\smalpha '-\smalpha \smbeta '+\smalpha \smbeta \smbeta '+\smalpha
\smalpha ' \smbeta '-\smbeta \smalpha ' \smbeta ') z_1\right]
& \kern-15pt z_4^2
\end{array}
\end{eqnarray}
\begin{eqnarray}
\label{singularCubic2}
\begin{array}{cccrc}
\sC_2 &=&& \left[(-\smbeta \smalpha '+\smalpha \smbeta \smalpha '+\smalpha
\smbeta '-\smalpha \smbeta \smbeta '-\smalpha \smalpha ' \smbeta
'+\smbeta \smalpha ' \smbeta ') z_2\right] \kern-10pt & z_3^2 \\[3pt] 
&&+& \left[\left(\smbeta (\smalpha ')^2-\smalpha \smbeta (\smalpha')^2 
-\smalpha^2 \smbeta '+\smalpha \smbeta \smbeta
  '+\smalpha^2 \smalpha ' \smbeta '-\smbeta \smalpha '
  \smbeta '\right) z_2 \right. & \\ 
&&& \ +\left.\left(-\smalpha^2
  \smalpha '+\smbeta \smalpha '+\smalpha (\smalpha ')^2-\smbeta
  (\smalpha ')^2-\smalpha \smbeta '+\smalpha^2 \smbeta
  '\right) z_4\right] 
& \kern-10pt z_1 z_3 \\[3pt] 
&&+& \left[\left(\smalpha^2 \smalpha
'-\smalpha \smbeta \smalpha '-\smalpha (\smalpha ')^2+\smalpha
\smbeta (\smalpha ')^2+\smalpha \smalpha ' \smbeta '-\smalpha^2
\smalpha ' \smbeta '\right) z_4\right] 
& \kern-15pt z_1^2
\end{array}
\end{eqnarray}
All the cubics in this pencil contain all $7$ lines.  If there is a
non-singular cubic containing all $7$ lines, then all cubics containing
the $7$ lines lie in the pencil. All the non-singular surfaces in the
pencil are projectively isomorphic. 

All non-singular cubics arise in this way.
\end{theorem}

\begin{proof}
  We have proved all of this already, except the explicit formulae.

  One approach to proving this is brute force. Write down the $18
  \times 20$ matrix corresponding to the 18 equations in 20
  unknowns. One can then compute its kernel in order to find the two
  equations. This is not as tedious as one might expect; it can be
  done by hand, and is the work of moments for a computer algebra
  system. We have included the formulae for completeness, but will not
  use them directly, so we will omit the details.

  It is interesting, however, to understand the general form of these
  equations, and this is something we will take advantage of.

  We have seen that the projective transformations
\begin{equation}\label{uv}
\phi(u,v):[z_1,z_2,z_3,z_4] \longrightarrow [u z_1, v z_2, u z_3, v
  z_4]
\end{equation}
will preserve the $7$ lines. We therefore look for cubic surfaces
which are linear in $z_1$, $z_3$ and quadratic in $z_2$, $z_4$. In
other words cubics of the form:
\begin{equation}
\label{ruledCubicEqn}
(a z_1 + b z_3) z_2^2 + (c z_1 + d z_3) z_2 z_4 + (e z_1 + f z_3)
z_4^2 = 0
\end{equation}
The justification for considering such surfaces is that they will
always contain $b_5$ and $b_6$ and will be invariant under
$\phi(u,v)$.

Suppose that $(0,w_2,0,w_4)$ and $(w_1, 0,w_3,0)$ are points on $b_5$
and $b_6$. The general point on the line between these points is:
\[[\lambda w_1, \mu w_2, \lambda w_3, \mu w_4]\] with $\lambda$, $\mu$
in $\CC$. When we put the coordinates of this point into equation
\eqref{ruledCubicEqn}, we get a common factor of $\lambda \mu^2$. Hence
the line lies on this cubic surface if and only if a single point on
the line away from $b_5$ and $b_6$ does.

If we choose a generic plane transverse to $b_5$ and $b_6$ it will
intersect the lines $a_1$, $a_2$, $a_3$, $a_4$, $c_{56}$ in 5
points. Plugging the coordinates of these intersection points into
equation \eqref{ruledCubicEqn} we get $5$ linear equations in the 6
unknowns $a,b,\ldots,f$. So there is a cubic surface of the given form
that contains all $7$ lines.\end{proof}

We have chosen this presentation of the classification of cubic
surfaces because it yields the following classification for twistor
lines on cubic surfaces.

\begin{theorem}
  For a generic set of $5$ points lying on a $2$-sphere in $S^4$,
  there exists a one-parameter family of projectively isomorphic but
  conformally non-isomorphic non-singular cubic surfaces with $5$
  twistor lines corresponding to the $5$ points.

  All cubic surfaces with $5$ twistor lines arise in this way.  Given
  such a surface, one can label the twistor lines $a_1$, $a_2$, $a_3$,
  $a_4$, $c_{56}$ and the two transversals $b_5$, $b_6$.

  One can associate a real invariant $\xi$ to a labelled cubic surface
  with five twistor lines in such a way that labelled cubic surfaces
  containing $5$ twistor lines are conformally isomorphic if and only
  if the points on the sphere are conformally isomorphic and the
  values for $\xi$ are equal.
\label{xiDefinition}
\end{theorem}

\begin{proof}
  A $2$-sphere in $S^4$ lifts to two projective lines $b$ and $jb$ in
  $\CP^3$. The choice of $5$ points on the sphere determines $5$
  collinear lines in $\CP^3$. It follows from above that there exists
  a one-parameter family of cubics containing all $7$ lines. We shall
  show later that if the $5$ points are chosen generically then the
  general cubic in this family is non-singular. This being the case,
  we can label the twistor fibres $a_1$, $a_2$, $a_4$, $a_4$,
  $c_{56}$, and the transversals $b=b_5$, $jb=b_6$.

  The bijection $b_6\to b_5$ is determined by its action on 3 points,
  so in the coordinates used in our study we have $\jj[z_1,0,z_3,0] =
  [0,\overline{z_1},0,\overline{z_3}]$. The action of $\jj$ on all of
  $\CP^3$ follows by anti-linearity.
 
  Since $\jj$ maps the intersection of $a_i$ and $b_5$ to the
  intersection of $a_i$ with $b_6$, it follows that $\smalpha =
  \overline{\smalpha'}$ and $\smbeta = \overline{\smbeta '}$ in
  \eqref{crossRatioInvariants}.

  In general, a projective transformation \eqref{uv} of $\CP^3$ which
  fixes all $7$ labelled lines will not correspond to a conformal
  transformation of $S^4$. It will do so if and only if it preserves
  $\jj$. This will be the case if and only if $|u|=|v|$.

  The general cubic surface containing all $7$ lines is given by a
  linear combination of $\sC_1$ and $\sC_2$ as defined in equation
  (\ref{singularCubic1}) and (\ref{singularCubic2}). Given such a
  cubic, define $M \in \RR$ to be the coefficient of $z_1 z_4^2$ and
  $N \in \RR$ to be the coefficient of $z_3 z_2^2$. Define
  $\xi=|M/N|$.

  We need to check that neither $M$ nor $N$ is zero. We know that $M$
  is a non-zero multiple of the corresponding coefficient in the
  polynomial (\ref{singularCubic2}). Suppose that this coefficient
  were equal to zero. This would mean that any cubic surface
  containing the $7$ lines would depend only linearly upon $z_1$ since
  this is the only non-linear term in $z_1$ in either
  (\ref{singularCubic2}) or (\ref{singularCubic1}). This would mean
  that the cubic was ruled and hence singular. Similarly, we see that
  $N$ is non-zero.

  By construction, $\xi$ is invariant under transformations $\phi(u,v)$
  with $|u|=|v|$. Thus it is well defined solely in terms of the cubic
  surface and the labelling. Since $\xi$ changes in proportion to
  $u/v$, $\xi$ will always distinguish conformally inequivalent
  surfaces.
\end{proof}

\begin{corollary}
  Given the $27$ lines on a non-singular cubic surface there is an
  algorithm to determine whether it has a twistor structure such that
  $5$ of the $27$ are twistor fibres.
\end{corollary}

\begin{proof}
  Run through all pairs of skew lines, compute the cross ratios of the
  intersection points and check whether $\smalpha =
  \overline{\smalpha'}$ and $\smbeta = \overline{ \smbeta'}$
\end{proof}

\bigbreak

We can summarize by saying that a cubic surface depends up to
projective transformation upon a choice of $4$ complex parameters
$\smalpha, \smbeta, \smalpha ', \smbeta '$ determined by the cross
ratios of the line intersection points. These $4$ complex parameters
do depend upon a labelling of the lines in the cubic --- so we have an
action of the graph isomorphism group of the Schl\"afli graph on the
space of such parameters. This is the Weyl group $W(E_6)$ of the
exceptional Lie group $E_6$. Thus the moduli space of cubic surfaces
up to projective isomorphism is given by an open subset of $\CC^4$
quotiented by $W(E_6)$. For more details, we refer the reader to
\cite{naruki} and \cite{ACT}.

In the case of conformal isomorphism classes of cubic surfaces with 5
twistor lines we have a choice of 2 complex parameters $\smalpha, \smbeta$
and one real parameter $\xi$. In addition we have a choice of
labelling for the $5$ twistor lines and a labelling of the two lines
collinear to all the twistor lines. So the moduli space of cubic
surfaces with $5$ twistor lines up to conformal isomorphism is given by
an open subset of $\CC^2 \times \RR$ quotiented by $S_5 \times \ZZ_2$.

In both cases we can write down an explicit equation for a cubic
surface with given values for the parameters by choosing appropriate
multiples of polynomials (\ref{singularCubic1}) and
(\ref{singularCubic2}).

\section{Identifying non-singular cubic surfaces with $5$ 
twistor lines}

Modern treatments of the classification of cubic surfaces usually
state that non-singular cubic surfaces are given by blowing up 6
points in $\CP^2$ in general position, the latter meaning that no 3
points are collinear and that the $6$ points do not all lie on a
conic.

This perspective highlights the intrinsic complex geometry of the
cubic surfaces --- it ostensibly describes cubic surfaces up to
biholomorphism rather than up to projective transformation. However,
these two classifications are equivalent. This is guaranteed by the
fact that any automorphism of a smooth hypersurface of $\CP^n$
($n\geqslant3$) of degree $d \ne n+1$ is induced by a projective
transformation. This in turn follows from the general correspondence
between maps to projective space and sections of complex line bundles,
see \cite{griffithsAndHarris}.

Because we are interested in the classification up to conformal
transformation, we have emphasized the embedding of the cubic surface
om $\CP^3$. Let us review the connection between this and the
intrinsic geometry.

Given a cubic surface $\sC$, let $\psi_1$ be the biholomorphism
$\CP^1\to b_5$ mapping $0$, $1$, $\infinity$ to the intersection points
of $b_5$ with $a_1$, $a_2$, $c_{56}$ respectively. Similarly define
$\psi_2:\CP^1\to b_6$ by sending $0$, $1$, $\infinity$ to points on
$a_1$, $a_2$, $c_{56}$. One can now define a rational map $\psi: \CP^1
\times \CP^1\to\sC$ by defining $\psi(z_1,z_2)$ to be the intersection
point of the line containing $\psi_1(z_1)$, $\psi_2(z_2)$ with the
surface $\sC$.

The map $\psi$ will be well defined for general points $(z_1,z_2)$ in
$\CP^1 \times \CP^1$. If we incorporate the multiplicity of the
intersection into our definition, it is clear that we have a map
$\psi$ well defined at all points except: $(0,0)$, $(1,1)$,
$(\smalpha,\smalpha ')$, $(\smbeta,\smbeta ')$, $(\infinity,\infinity)$
where the $\smalpha$'s and $\smbeta$'s are the cross-ratio invariants
defined earlier. These points correspond to the lines $a_1$, $a_2$,
$a_3$, $a_4$, $c_{56}$ respectively, and are indicated in Figure
\ref{fig:rationalMap}.
It turns out that $\psi$ extends to a biholomorphism from $\CP^1
\times \CP^1$ blown up at these five points to the cubic surface
$\sC$.

\medskip
\begin{figure}[h]
\scalebox{1.25}{\includegraphics{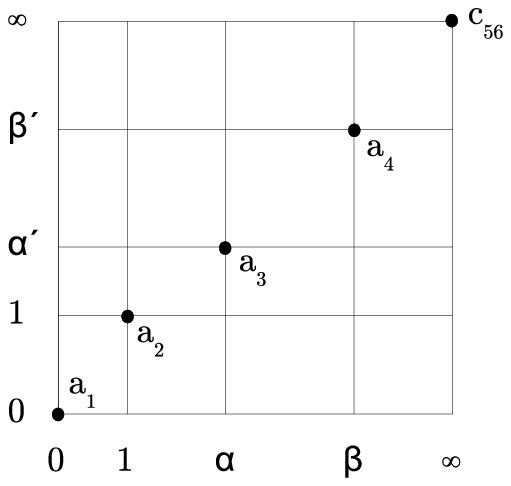}} 
\centering 
\caption {The five points to blow up on $\CP^1\times\CP^1$}
\label{fig:rationalMap}
\end{figure}
\bigskip 

Now, $\CP^1 \times \CP^1$ can be thought of as $\CP^2$ with two points
at infinity blown up and then the line at infinity blown down. This
allows us to think of the blow up of $\CP^1 \times \CP^1$ at $5$
points as being the blow up of $\CP^2$ at $6$ points corresponding to
$a_1$, $a_2$, $a_3$, $a_4$, $a_5$, $a_6$. Now, $c_{56}$ corresponds to
the line at infinity.

To be very concrete, blowing up the two points $[1,0,0]$ and $[0,1,0]$
at infinity and then blowing down the proper transform of the line at
infinity is given by the rational map $(z_1,z_2) \to (z_1,z_2)$. The
left hand side should be viewed as giving inhomogeneous coordinates
for $\CP^2$, the right hand side as giving inhomogeneous coordinates
for each factor of $\CP^1 \times \CP^1$. We define a rational map
$\tilde{\psi}$ from $\CP^2$ to our cubic by $\tilde{\psi}(z_1,z_2) =
\psi(z_1,z_2)$.

Since any four points in general position in $\CP^3$ are projectively
equivalent, we see that a choice of $6$ points to blow up corresponds
to the $4$ cross ratios $\smalpha$, $\smalpha'$, $\smbeta$,
$\smbeta'$.

Notice that the lines on the cubic surface are easily understood in
terms of the blow-up picture. The $a$ lines correspond to the points
that have been blown up. The line $c_{ij}$ corresponds to the straight
line in $\CP^2$ passing through the points $a_i$ and $a_j$. The line
$b_i$ corresponds to the conic passing through all the blown-up points
except $a_i$. One can immediately read off the intersection properties
of all the lines when they are thought of in terms of this
picture. This gives a particularly nice way of remembering the
structure of the Schl\"afli graph.

This intrinsic view of cubic surfaces allows us to tie up a loose end
we left dangling in the previous section. Recall that given two cubics
$\sC_1$ and $\sC_2$ containing the lines $a_1$, $a_2$, $a_3$, $a_4$,
$c_{56}$ we constructed a rational map $\Phi: \CP^3\to\CP^3$ sending
$\sC_1$ to $\sC_2$. We claimed that this map could was in fact a
biholomorphism. Identifying each of $\sC_1$ and $\sC_2$ with $\CP^2$
blown up at six points, we see that $\Phi$ restricted to $\sC_1$ is
essentially the identity --- hence it is certainly a
biholomorphism. 

The most important feature of this intrinsic view of cubic surfaces
from the perspective of the twistor geometry is the criteria it gives
for determining whether a cubic surface is non-singular. The blow-up
of $\CP^2$ at 6 points is obviously smooth, so if we can find a cubic
surface corresponding to the 6 points, that cubic surface will be
non-singular. To construct a cubic surface given 6 points in general
position, one considers the vector space $\mathcal C$ of cubic
\emph{curves} in $\CP^2$ that pass through all 6 points. This space
will be $4$-dimensional so long as no 6 points lie on a conic and no 3
points lie on a line. One then defines a rational map sending a point
$z\in\CP^2$ to the projectized dual space $\PP(\mathcal{C}^*)$ by
mapping a cubic polynomial to its value at $z$. This rational map is a
biholomorphism of the blow up at 6 points to a cubic surface in
$\PP(\mathcal{C}^*)$. This result was first discoved by Clebsch in
\cite{clebsch}. Details of the proof can be found in
\cite{griffithsAndHarris}.

We saw in the previous section that given $5$ points lying on a 2-sphere
in $S^4$ we can find a family of cubic surfaces with $5$ twistor lines
corresponding to these five points. It follows from the discussion
above that if the $5$ points on the 2-sphere are chosen in general
position then the cubic surfaces will be non-singular. We would like
to identify more clearly what ``in general position'' actually means
in this case. 

\begin{theorem}
Given $5$ points lying on a 2-sphere in $S^4$, there is a non-singular
cubic surface with $5$ twistor lines corresponding to these points if
and only if no $4$ of the points lie on a circle.
\end{theorem}

\begin{proof}
  There are two lines in $\CP^3$ lying above $S^4$ under the twistor
  correspondence. Label one of them $b_5$ and the other $b_6$.
  
  To each of the five points on $b_5$, there is a unique twistor line
  over that point. We label these lines arbitrarily as $a_1$, $a_2$,
  $a_3$, $a_4$ and $c_{56}$.

  Three distinct points on a $2$-sphere are conformally equivalent,
  and so always in general position. We choose an inhomogeneous
  coordinate $z_1$ for $b_5$ and $z_2$ for $b_6$ by requiring that the
  intersections of $a_1$, $a_2$, $a_3$, $a_4$ and $c_{56}$ with $b_5$
  are given by $0$, $1$, $\smalpha$, $\smbeta$ and
  $\infinity$. Similarly we choose an inhomogeneous coordinate $z_2$
  for $b_6$ such that the intersection points are $0$, $1$,
  ${\overline \smalpha}$, ${\overline \smbeta}$ and $\infinity$.

  We now have an unambiguously defined rational map $\phi$ from
  $\CP^2$ to $b_5 \times b_6$ given in inhomogenous coordinates by
  $\phi(z_1,z_2) = (z_1,z_2)$.  Corresponding to each of the lines
  $a_1$, $a_2$, $a_3$, $a_4$ and $c_{56}$ we can define points $a_1'$,
  $a_2'$, $a_3'$, $a_4'$ and $c_{56}'$ in $b_5 \times b_6$ given by
  sending a line to its intersection point with each $b_i$. We can
  then define six points in $\CP^2$ as follows:
\begin{eqnarray*}
 A_1 & = &\phi^{-1}( a_1') = [0,0,1] \\
 A_2 & = &\phi^{-1}( a_2') = [1,1,1] \\
 A_3 & = & \phi^{-1}( a_3' ) = [\smalpha,{\overline \smalpha},1] \\
 A_4 & = & \phi^{-1}( a_4' ) = [\smbeta,{\overline \smbeta},1] \\
 A_5 & = & [1,0,0] \\
 A_6 & = & [0,1,0] 
\end{eqnarray*}
Note that $c_{56}$ corresponds to the line at infinity in
$\CP^2$. $A_5$ and $A_6$ correspond to the two lines at infinity in
$b_5 \times b_6$. The setup is precisely as summarized in
Figure~\ref{fig:rationalMap}.

The point we are making is that these points in $\CP^2$ are determined
entirely by the $5$ points on the sphere and the choice of labelling:
we do not need there to be a non-singular cubic through the five
twistor lines in order to construct the $A_i$.

  The blow up of $\CP^2$ at the $A_i$ corresponds to a smooth cubic if
  and only if these $A_i$ are in general position (meaning no three
  collinear and no conic through all $6$). This cubic must then be
  biholomorphic to one of the cubic curves in the pencil generated by
  \eqref{singularCubic1} and \eqref{singularCubic2}. We deduce that
  there is a smooth cubic with $5$ twistor lines corresponding to the
  five points on $S^2$ if and only if these six points $A_i$ in
  $\CP^2$ are in general position.

  $A_1$, $A_2$ and $A_3$ are collinear if and only if $\smalpha = {\overline
    \smalpha}$. This is equivalent to saying that $0$, $1$, $\smalpha$
  and $\infinity$ all lie on the real line. In invariant terms this is
  equivalent to requiring that $a_1$, $a_2$, $a_3$ and $c_{56}$ all
  lie on a circle in $S^2$.
  
  We deduce that there is a smooth cubic corresponding to the five
  points on $S^2$ only if no four of the points lie on a circle.

  It is a simple calculation to check that the condition that no four
  points lie on a circle implies that no three of the $A_i$ lie on a
  line. We also need to confirm that the same condition implies that
  there is no conic through all $6$ of the $A_i$.

  Suppose for a contradiction that there is such a conic and so the
  $A_i$ form an ``inscribed hexagon''. Pascal's theorem implies the
  intersection points
\[ A_1A_2\cap A_4A_5,\quad A_2A_3\cap A_5A_6,\quad A_3A_4\cap A_6A_1\]
are collinear. These points can be computed using the vector cross
product; the first is $(A_1\times A_2)\times(A_4\times A_5)$ with
a slight abuse of notation. The collinearity condition is then
\[ 
\det\left(\!\begin{array}{ccc} 
\overline\smbeta & \overline\smbeta & 1 \\ 
\smalpha-1 & \overline\smalpha-1 & 0\\ 
0 & \smbeta\overline\smalpha-\smalpha\overline\smbeta & 
\smbeta-\smalpha\end{array}\!\right)=0,\]
which gives 
\[
|\smalpha|^2(\smbeta-\overline\smbeta)-
|\smbeta|^2(\smalpha-\overline\smalpha)+
\alpha\overline\smbeta-\overline\smalpha\smbeta=0.\]
But this is easily seen to be exactly the condition that
$0,1,\smalpha,\smbeta$ lie on a circle in $\CC.$
\end{proof}

On first reading this proof one may wonder where the asymmetry between the $a_i$ and $c_{56}$ arises. It can be traced directly to the choice to associate $c_{56}$ with the points $z_1=\infinity$ and $z_2=\infinity$.

For a coordinate-free explanation of why the cubic must be singular if
four of the $5$ points in $S^2$ lie in a circle ${\Gamma}$, recall
that by by \cite[Theorem~3.10]{salamonViaclovsky} that there is a
quadric surface $\sQ$ in $\CP^3$ containing $\pi^{-1}(\Gamma)$ (where
$\pi$ is the twistor projection). Suppose that there is also
non-singular cubic $\sC$ for which the twistor fibres
$a_1,a_2,a_3,a_4$ from part of a double six $(a_i,b_j)$. In this case,
$\sC$ must be the only cubic containing the double six. But each $b_j$
intersects at least three of $a_1,a_2,a_3,a_4$ and therefore lies in
$\sQ$. The latter must now contain $a_5,a_6$ as well. But then the
union of $\sQ$ and any plane is a cubic containing the double six,
which is a contradiction.

As an application of our theorem, we observe that the well known
Clebsch diagonal surface does not have $5$ twistor lines irrespective
of the twistor structure $\jj$ one places on $\CP^3$. The Clebsch
diagonal surface is the complex surface in $\CP^4$ defined by the two
equations
\[\begin{array}{rcl}
 z_1^3 + z_2^3 + z_3^3 + z_4^3 + z_5^3 = 0\\[3pt]
 z_1 + z_2 + z_3 + z_4 + z_5 = 0.
\end{array}\]
It is biholomorphic to the surface in $\CP^3$ given by the single
equation:
\[ z_1^3 + z_2^3 + z_3^3 + z_4^3 = (z_1 + z_2 + z_3 + z_4)^3,\]
and is the only cubic surface with symmetry group $S_5$. It has the
nice property that all 27 lines on the cubic surface are real
lines. This immediately means that it admits no twistor structure $\jj$
such that it has five twistor lines. Simply note that the cross ratios
of all the intersection points on the lines must be real. Therefore
any four points on one of its lines must lie on a circle when that
line is viewed as the Riemann sphere.

\section{The Fermat cubic}

Having found a large family of cubic surfaces with $5$ twistor lines we
would like to ask if there are any particularly nice examples. In
particular what is the most symmetrical cubic surface with $5$ twistor
lines?

A conformal transformation of $S^4$ that induces a symmetry of a cubic
surface with $5$ twistor lines must leave the 2-sphere image of $b_5$
and $b_6$ fixed. If the conformal transformation leaves the image of
the $5$ twistor lines fixed, then the associated projective
transformation must swap $b_5$ and $b_6$. Otherwise the conformal
transformation must permute the $5$ points on the 2-sphere.

Therefore let us first choose the most symmetrical arrangement of 5
points on a 2-sphere no four of which lie on a circle. If we have a
rotation of the sphere that permutes $n$ of the points then those
points must all lie on a circle. So $n \leqslant3$. So any rotation
fixes at least two points. Either those two points lie on the axis of
rotation, or the rotation is a rotation through 180 degrees and the
fixed points all lie on a circle. We deduce that the largest possible
symmetry group for the five points is $\ZZ_3 \times \ZZ_2$ and, up to
conformal transformation of the 2-sphere, we can assume that our five
points are $0$, $\infinity$ and the three cube roots of unity. The
$\ZZ_3$ action rotates the three cube roots into each other. The
$\ZZ_2$ action swaps $0$ and $\infinity$.

Setting $\smalpha$ and $\smbeta$ to be complex cube roots of unity and
$\smalpha '$ and $\smbeta '$ to be their conjugates, polynomials
\eqref{singularCubic1} and \eqref{singularCubic2} simplify to:
\[\begin{array}{l}
-3 i \sqrt{3} \left(z_2 z_3^2-z_1^2 z_4\right),\\[5pt]
3i\sqrt{3} \left(-z_2^2 z_3+z_1 z_4^2\right).
\end{array}\]
We now wish to choose a linear combination of these polynomials that
will remain fixed under the transformation that swaps the lines $b_5$
and $b_6$. This corresponds to the projective transformation $z_1
\mapsto z_2$, $z_2 \mapsto z_1$, $z_3 \mapsto z_4$, $z_4 \mapsto z_3$.

Hence there is, up to conformal transformation, a unique non-singular
cubic surface with $5$ twistor lines and symmetry group $\ZZ_3 \times
\ZZ_2 \times \ZZ_2$. It is defined by
\begin{equation}
\label{symmetricalCubicSurface}
z_1 z_4^2 + z_2 z_3^2 - z_3 z_2^2 - z_4 z_1^2 = 0.
\end{equation}
One can make a conformal transformation (corresponding to using the
cube roots of $-1$ rather than those of $1$) to replace the two minus
signs with plus signs. In any case, it is projectively, but not
conformally, equivalent to a familar example: the Fermat cubic, which
is defined by the equation
\[ z_1^3 + z_2^3 + z_3^3 + z_4^3 =0. \]
One can prove that these surfaces are projectively equivalent by
calculating the cross ratio invariants we defined earlier. This
approach allows one to write down an explicit linear transformation
sending the Fermat cubic to the surface
(\ref{symmetricalCubicSurface}).

A more pleasing approach is to use the symmetries of the Fermat cubic
to deduce that there must be some twistor structure that gives it five
twistor lines. To see how this is done, first choose a complex cube
root of unity $\smomega$ and label $7$ of the lines on the Fermat cubic as
follows:

\bigbreak

\begin{center}
\begin{tabular}{ll} \toprule
Label\phantom{mmm} & Line \\ 
\midrule $b_5$ & $z_1+ \smomega z_2=z_3+ \smomega^2 z_4=0$
\\ $b_6$ & $z_1+ \smomega^2 z_2=z_3+ \smomega z_4=0$ \\ $a_1$ & $z_1 +
\smomega z_2=z_3+ \smomega z_4=0$ \\ $a_2$ & $z_1 + z_4 = z_2 + z_3 = 0$
\\ $a_3$ & $z_1 + \smomega z_4 = z_2 + \smomega z_3 = 0$ \\ $a_4$ & $z_1 +
\smomega^2 z_4 = z_2 + \smomega^2 z_3 = 0$ \\ $c_{56}$ & $z_1 + \smomega^2
z_2=z_3+\smomega^2 z_4=0$ \\ \bottomrule
\end{tabular}
\end{center}
\bigbreak

Consider the symmetry of the cubic given by
\[(z_1,z_2,z_3,z_4) \longmapsto (z_1, z_2, \smomega z_3, \smomega z_4).\]
This generates a $\ZZ_3$ action that fixing the lines $b_5$, $b_6$,
$a_1$ and $c_{56}$ and permuting $a_2$, $a_3$, $a_4$. Thus we have a
$\ZZ_3$ symmetry of $b_5$ fixing the intersection points with $a_1$
and $c_{56}$ and permuting the intersection points with $a_2$, $a_3$
and $a_4$. Therefore these $5$ points on $b_5$ are conformally
equivalent to the points $0$, $\infinity$, $1$, $\smomega$ and
$\overline{\smomega}$ on the Riemann sphere. The same applies to the $5$
intersection points with $b_6$. This implies that the invariants
(\ref{crossRatioInvariants}) satisfy $\smalpha = \overline{\smbeta}$
and $\smalpha' = \overline{ \smbeta'}$

Now consider the symmetry:
\[ (z_1,z_2,z_3,z_4) \longmapsto (z_3, z_4, z_1, z_2).\]
This swaps $b_5$ and $b_6$, swaps $a_3$ and $a_4$ and fixes $a_1$,
$a_2$ and $c_{56}$. We deduce that:
\[
\smalpha = ( c_{56}, a_1; a_2, a_4)_{b_5} = (c_{56}, a_1; a_2,
a_3)_{b_6} = \smbeta' = \overline{\smalpha ^ \prime}.
\]
Thus the cross ratios of the intersection points on $b_5$ and $b_6$
are the same as the cross ratios for the intersection points on the
cubic surface (\ref{symmetricalCubicSurface}). We conclude that the
cubic given by (\ref{symmetricalCubicSurface}) is projectively
isomorphic to the Fermat cubic.

\begin{theorem} 
The set of twistor structures on $\CP^3$ for which the Fermat cubic
has $5$ twistor lines is a real $1$-manifold with 54 components.
\end{theorem}
\begin{proof}
  The surface (\ref{symmetricalCubicSurface}) has $12$ conformal
  symmetries. The Fermat cubic has symmetry group $S_4 \times \ZZ_3
  \times \ZZ_3 \times \ZZ_3$ given by permutations of the coordinates
  and by multiplying various coordinates by cube roots of unity. Thus
  there are $4! \cdot 27/12 = 54$ twistor structures on $\CP^3$ such
  that the Fermat cubic is isomorphic to surface
  (\ref{symmetricalCubicSurface}) with the standard twistor
  structure. We can then vary the invariant $\xi$ to get a
  one-parameter family of conformally inequivalent twistor structures.

  We need to check that there are no other twistor structures that
  give the Fermat cubic five twistor lines. We we gave an algorithm to
  do this eariler: run through all pairs of skew lines and compute
  cross ratios. We can speed this up significantly using the
  symmetries of the Fermat cubic. The general line on the Fermat cubic
  is \[z_i + \smeta_1 z_j = z_k + \smeta_2 z_l = 0\] where
  $\{i,j,k,l\}$ is a permutation of $\{1,2,3,4\}$ and $\smeta_1$ and
  $\smeta_2$ are cube roots of unity. So given two skew lines on the
  Fermat cubic, using the cubic's symmetries we can assume that the
  first line is:
\[ z_1 + z_2 = z_3 + z_4 = 0 \]
and the second line is one of:
\[\begin{array}{c}
 z_1 + \smeta_1 z_2 = z_3 + \smeta_2 z_4 = 0,\\[3pt]
 z_1 + \smeta_1 z_3 = z_2 + \smeta_2 z_4 = 0.
\end{array}\]
In the first case there is a $\ZZ_3$ symmetry preserving both lines
--- so if we have $5$ twistor lines it will be one of the cases
already considered. In the second case we can further assume that
$\smeta_1=1$ and, since the lines are skew, $\smeta_2\ne1$. Therefore
we just need to show that the cross ratios of the intersection points
of the 5 lines on the Fermat cubic meeting $z_1 + z_2 = z_3 + z_4 = 0$
and $z_1 + z_3 = z_2 + \smomega z_4 = 0$ do not form complex conjugate
pairs. This is easily done.
\end{proof}

There is a lot more one could ask about the twistor geometry of the
Fermat cubic. For example: what is the topology of its discriminant
locus? How does this vary as one varies the choice of twistor
structure? We will consider these questions in another paper.

\medskip

\noindent{\textbf{Acknowledgment.} This paper develops material from
  the second author's PhD thesis \cite{povero}. A major debt of
  gratitute is due to the Politecnico di Torino and to colleagues in
  its Department of Mathematical Sciences.

\small

\bibliographystyle{acm}
\bibliography{APS}

\bigskip

\begin{flushleft}

\textbf{AMS Subject Classification: 53C28, 14N10, 53A30} 

\bigskip

John ARMSTRONG\\
Department of Mathematics, King's College London\\
Strand, London WC2R 2LS, UK\\
e-mail: \texttt{johnarmstrong5@googlemail.com}\\[2ex]

Massimiliano POVERO\\
Dipartimento di Matematica, Universit\`a di Torino\\
Via Carlo Alberto 10, 10123 Torino, ITALIA\\
e-mail: \texttt{massimiliano.povero@polito.it}\\[2ex]

Simon SALAMON\\ 
Department of Mathematics, King's College London\\
Strand, London WC2R 2LS, UK\\
e-mail: \texttt{simon.salamon@kcl.ac.uk}

\end{flushleft}

\normalsize
\label{\thechapter:lastpage}

\enddocument